\pdfoutput=1
\documentclass[11pt]{amsart}
\usepackage[paperheight=279mm,paperwidth=18cm,textheight=26cm,textwidth=14cm]{geometry}
\usepackage{amsfonts,amssymb,amsmath,amsthm}
\usepackage{mathtools}
\usepackage{esint}
\usepackage{microtype}
\usepackage[T1]{fontenc}
\usepackage[utf8]{inputenc}
\usepackage[unicode]{hyperref}
\hypersetup{
bookmarks=true,
colorlinks=true,
citecolor=[rgb]{0,0,0.75},
linkcolor=[rgb]{0,0,0.75},
urlcolor=blue,
pdfpagemode=UseNone,
pdfstartview=FitH,
pdfdisplaydoctitle=true,
pdftitle={Variational estimates for singular integrals on SHT},
pdfauthor={Pavel Zorin-Kranich},
pdflang=en-US
}

\usepackage[style=alphabetic]{biblatex}
\numberwithin{equation}{section}
\newtheorem{theorem}[equation]{Theorem}
\newtheorem{corollary}[equation]{Corollary}
\newtheorem{lemma}[equation]{Lemma}
\newtheorem{proposition}[equation]{Proposition}
\theoremstyle{definition}
\newtheorem{definition}[equation]{Definition}

\theoremstyle{remark}
\newtheorem{remark}[equation]{Remark}
\newtheorem{example}[equation]{Example}

\newcommand{\R}{\mathbb{R}}
\newcommand{\Z}{\mathbb{Z}}
\newcommand{\N}{\mathbb{N}}
\newcommand{\DD}{\mathbb{D}}
\newcommand{\EE}{\mathbb{E}}
\newcommand{\CC}{\mathbb{C}}
\newcommand{\dif}{\mathrm{d}}
\newcommand{\sparse}{\mathcal{S}}

\newcommand{\cD}{\mathcal{D}}

\newcommand{\calD}{\mathcal{D}}
\newcommand{\calDp}{\mathcal{D}^{2}_{\subseteq}}

\newcommand{\calQ}{\mathcal{Q}}
\newcommand{\calS}{\mathcal{S}}

\DeclarePairedDelimiter\abs{\lvert}{\rvert}
\DeclarePairedDelimiter\meas{\lvert}{\rvert}
\DeclarePairedDelimiter\norm{\lVert}{\rVert}
\DeclarePairedDelimiter\Set\{\}

\DeclarePairedDelimiter\ceil{\lceil}{\rceil}

\def\<{\left\langle}
\def\>{\right\rangle}

\newcommand{\diam}{\mathrm{diam}}
\newcommand{\dist}{\mathrm{dist}}

\newcommand{\V}[3][r]{V^{#1}( #2 : #3)}
\newcommand{\hV}[3][r]{\dot{V}^{#1}( #2 : #3)}
\newcommand{\jump}[3][\lambda]{\mathcal{J}_{#1}(#2 : #3)}
\newcommand{\scale}{k}

\newcommand{\one}{\mathbf{1}}
\newcommand{\Nop}{\mathcal{N}}
\newcommand{\Bop}{\mathcal{B}}

\newcommand{\mean}[2][Q]{\<#2\>_{#1}}
\newcommand*{\supp}{\mathrm{supp}}

\def\clap#1{\hbox to 0pt{\hss#1\hss}}

\newcommand*{\Cref}[1]{C_{\eqref{#1}}}

\begin{document}

\title{Variational truncations of singular integrals on spaces of homogeneous type}
\author{Pavel Zorin-Kranich}
\address{University of Bonn\\ Mathematical Institute}
\maketitle
\begin{abstract}
We prove sharp weighted estimates for $r$-variations of averages and truncated singular integrals on suitable spaces of homogeneous type, including homogeneous nilpotent Lie groups.
\end{abstract}

\section{Introduction}
Throughout the article $(X,\rho,\mu)$ denotes an (Ahlfors--David) $D$-regular space of homogeneous type (definitions of these and other terms are given in Section~\ref{sec:notation}).
Variational estimates for the averaging operators 
\begin{equation}
\label{eq:At}
A_{t}f(x) = \mu(B(x,t))^{-1} \int_{B(x,t)} f(y) \dif\mu(y)
\end{equation}
on $X=\R^{D}$ have been introduced by Bourgain \cite{MR1019960}.
A comprehensive theory covering the full range of $L^{p}$ spaces and variational exponents $r$ both for averages \eqref{eq:At} and truncated singular integrals 
\begin{equation}
\label{eq:Tt}
T_{t}f(x) = \int_{\rho(x,y)>t} K(x,y) f(y) \dif\mu(y)
\end{equation}
has been developed by a number of authors \cite{MR1996394,MR1645330,MR2434308}.
Some of their estimates have been extended to weighted $L^{p}$ spaces in \cite{MR3283159,arxiv:1409.7120,arXiv:1511.05129}.

Sparse domination has been developed in \cite{MR3085756,arXiv:1501.05818,MR3625128} in order to simplify the proof of the $A_{2}$ theorem for Calder\'on--Zygmund (CZ) operators \cite{MR2912709}.
In a short period of time since 2015 this idea has been applied in many settings which go beyond CZ theory, and we are not going to survey these developments.
In the CZ setting it is by now well understood that sparse domination follows from suitable localized non-tangentional endpoint estimates; several abstract results formalizing this principle appeared in \cite{MR3484688,arXiv:1604.05506,arxiv:1612.09201}.
These techniques have been applied to $r$-variational estimates for truncated singular integrals in \cite{MR3065022} (smooth truncations) and \cite{arXiv:1604.05506} (sharp truncations).

In this article we extend these $r$-variational estimates to a class of non-convolution type singular integrals.
We formulate our results on classes of spaces of homogeneous type that include homogeneous nilpotent Lie groups.
In this setting we also obtain some sharp weighted inequalities for square functions and $r$-variation of averages.

Weighted estimates for $r$-variation of averaging operators \eqref{eq:At} have been obtained in \cite{arxiv:1409.7120}.
While in retrospect the methods of that article easily imply sparse domination of $r$-variations, the endpoint estimate for the jump counting function below requires a new ingredient.
\begin{theorem}
\label{thm:av}
Suppose that the space $X$ has the small boundary property.
Then the operators $f\mapsto \lambda^{-1} \sqrt{\jump{A_{t}f}{0<t<\infty}}$ are pointwise controlled by sparse operators uniformly in $\lambda>0$.
Moreover, for every $r>2$ the operator $f\mapsto \V{A_{t}f}{0<t<\infty}$ is pointwise controlled by sparse operators with constant $C_{X} r/(r-2)$.
\end{theorem}

Known estimates for sparse operators (see \cite{MR3778152} for $p>1$ and \cite{MR3455749} for $p=1$) yield the following consequences for the jump counting function (and similar ones for $r$-variations).
\begin{corollary}
\label{cor:av}
With the hypotheses of Theorem~\ref{thm:av}
\begin{align}
\label{eq:av:ap}
\norm{\lambda^{-1} \sqrt{\jump{A_{t}(f\sigma)}{t>0}}}_{L^{p}(w)}
&\lesssim_{X}
[w,\sigma]_{A_{p}}^{1/p} ([w]_{A_{\infty}}^{1/p'} + [\sigma]_{A_{\infty}}^{1/p}) \norm{f}_{L^{p}(\sigma)},
& 1<p<\infty,\\
\label{eq:av:weak-ap}
\norm{\lambda^{-1} \sqrt{\jump{A_{t}(f\sigma)}{t>0}}}_{L^{p,\infty}(w)}
&\lesssim_{X}
[w,\sigma]_{A_{p}}^{1/p} [w]_{A_{\infty}}^{1/p'} \norm{f}_{L^{p}(\sigma)},
& 1<p<\infty,\\
\label{eq:av:a1}
\norm{\lambda^{-1} \sqrt{\jump{A_{t}(f\sigma)}{t>0}}}_{L^{1,\infty}(w)}
&\lesssim_{X}
\frac{1}{\alpha-1} \int_{X} \abs{f} M_{L \log\log L (\log\log\log L)^{\alpha}} w,
& 1<\alpha<2.
\end{align}
\end{corollary}

Finally, we obtain $r$-variational estimates for a class of non-convolution type singular integral operators.

\begin{theorem}
\label{thm:sing}
Suppose that the space $X$ has the small boundary property and the metric $\rho$ satisfies the H\"older type condition
\begin{equation}
\label{eq:rho-holder}
\abs{\rho(x,z)-\rho(y,z)} \leq C \max (\rho(x,z),\rho(y,z))^{1-\eta} \rho(x,y)^{\eta}
\end{equation}
for some $0<\eta\leq 1$.

Let $K : X\times X \setminus \Set{\mathrm{diagonal}} \to \CC$ be a H\"older continuous CZ kernel, that is, a function that satisfies the size condition
\begin{equation}
\label{eq:size}
\sup_{x\neq y} \rho(x,y)^{D} \abs{K(x,y)} \leq 1,
\end{equation}
the smoothness estimate
\begin{equation}
\label{eq:smoothness}
\abs{K(x,y)-K(x',y)} + \abs{K(y,x)-K(y,x')}
\leq
\big( \frac{\rho(x,x')}{\rho(x,y)} \big)^{\eta} \rho(x,y)^{-D}
\end{equation}
for all $x,y$ with $\rho(x,y)\geq 2\rho(x,x')>0$, and the cancellation condition
\begin{equation}
\label{eq:cancel}
\int_{r<\rho(x,y)<R} K(x,y) \dif x = \int_{r<\rho(x,y)<R} K(x,y) \dif y = 0,
\quad 0<r<R<\infty.
\end{equation}
Consider the associated truncated singular integral operators \eqref{eq:Tt}.

Then for every $r>2$ the operator $f\mapsto \V{T_{t}f}{0<t<\infty}$ is pointwise controlled by sparse operators with norm $C_{X,\eta}r/(r-2)$.
\end{theorem}

\begin{corollary}
\label{cor:sing}
With the hypotheses of Theorem~\ref{thm:sing}
\begin{align}
\label{eq:sing:ap}
\norm{\V{T_{t}(f\sigma)}{t>0}}_{L^{p}(w)}
&\lesssim_{X}
\frac{r}{r-2} [w,\sigma]_{A_{p}}^{1/p} ([w]_{A_{\infty}}^{1/p'} + [\sigma]_{A_{\infty}}^{1/p}) \norm{f}_{L^{p}(\sigma)},
& 1<p<\infty,\\
\label{eq:av:weak-ap}
\norm{\V{T_{t}(f\sigma)}{t>0}}_{L^{p,\infty}(w)}
&\lesssim_{X}
\frac{r}{r-2} [w,\sigma]_{A_{p}}^{1/p} [w]_{A_{\infty}}^{1/p'} \norm{f}_{L^{p}(\sigma)},
& 1<p<\infty,\\
\label{eq:av:a1}
\norm{\V{T_{t}(f)}{t>0}}_{L^{1,\infty}(w)}
&\lesssim_{X}
\frac{r}{r-2} \frac{1}{\alpha-1} \int_{X} \abs{f} M_{L \log\log L (\log\log\log L)^{\alpha}} w,
& 1<\alpha<2.
\end{align}
\end{corollary}

It is known \cite[Theorem 2]{MR546295} that for every quasimetric $\rho$ there exist $\alpha>0$ and a metric $\tilde\rho\sim\rho^{\alpha}$ with the property \eqref{eq:rho-holder}.
More practically, the property \eqref{eq:rho-holder} holds for homogeneous metrics on homogeneous nilpotent Lie groups.

The conclusion of Theorem \ref{thm:sing} has been previously known for convolution type CZ kernels on $\R^{D}$ by \cite[Theorem A.1]{arXiv:1512.07523} (unweighted $L^{p}$ estimates), \cite[Theorem B]{MR1953540} (weak type $(1,1)$), and \cite{arXiv:1604.05506} (reduction of weighted to unweighted estimates).

For rough homogeneous kernels on $\R^{D}$ variational estimates have been proved in \cite[Theorem A]{MR1953540} (see also \cite{MR2434308} and \cite{MR3649476}).
Quantiative weighted estimates for rough homogeneous kernels have been recently obtained in \cite{MR3625128,arxiv:1612.09201} and extended to their maximal truncations in \cite{arxiv:1705.07397,arxiv:1706.09064}.
It would be interesting to extend these results to variational truncations and also homogeneous groups as in \cite{MR1757083}.

\section{Notation and tools}
\label{sec:notation}
\subsection{Spaces of homogeneous type}
\begin{definition}
A \emph{quasi-metric} on a set $X$ is a function $\rho : X\times X\to [0,\infty)$ such that $\rho(x,y)=0 \iff x=y$ that is symmetric and satisfies the quasi-triangle inequality
\begin{equation}
\label{qm}
\rho(x,y) \leq A_{0} (\rho(x,z) + \rho(z,y))
\quad\text{for all}\quad
x,y,z\in X
\end{equation}
with some $A_{0}<\infty$ independent of $x,y,z$.

A measure $\mu$ on a quasi-metric space $(X,\rho)$ is called \emph{doubling} if there exists $A_{1}<\infty$ such that
\begin{equation}
\label{db}
\mu(B(x,2r)) \leq A_{1} \mu(B(x,r))
\quad\text{for all}\quad
x\in X,r>0,
\end{equation}
A tuple $(X,\rho,\mu)$ consisting of a set $X$, a quasi-metric $\rho$, and a doubling measure $\mu$ is called a \emph{space of homogeneous type}.

A space of homogeneous type $(X,\rho,\mu)$ is called (Ahlfors--David) \emph{$D$-regular}, $D>0$, if there exist $0<c,C<\infty$ such that for all $x\in X$ and $r>0$ we have
\[
cr^{D} \leq \mu(B(x,r)) \leq Cr^{D}.
\]
A $D$-regular space necessarily has no atoms.
We say that a family $\calD$ of subsets of $X$ has the \emph{small boundary property} if there exist $\eta>0$ and $C_{3}<\infty$ such that for every $Q\in\calD$ and every $0<\tau\leq 1$
\begin{equation}
\label{eq:small-boundary}
\mu(\partial_{\tau \diam(Q)} Q) \leq C_{3} \tau^{\eta} \mu(Q),
\end{equation}
where
\begin{equation}
\label{eq:tau-boundary}
\partial_{\tau}(Q)
=
\Set{x\in Q : \dist(x,X\setminus Q)\leq\tau} \cup \Set{x\in X\setminus Q : \dist(x,Q)\leq\tau}.
\end{equation}
We say that $(X,\rho,\mu)$ has the small boundary property if the collection of all metric balls has the small boundary property.
\end{definition}
We denote the measure of a set $Q$ by $\meas{Q}=\mu(Q)$ and the average of a function $f$ over $Q$ by $\mean{f}=\meas{Q}^{-1} \int_{Q} f \dif\mu$.

\subsection{Dyadic cubes}
Filtrations on spaces of homogeneous type that closely resemble dyadic filtrations on $\R^{D}$ have been first constructed by Christ \cite{MR1096400}.
We recall their properties.
\begin{definition}
Let $(X,\rho,\mu)$ be a space of homogeneous type.
A \emph{system of dyadic cubes} $\calD$ with constants $\kappa>1$, $a_{0}>0$, $C_{1}<\infty$ consists of collections $\calD_{k}$, $k\in\Z$, of open subsets of $X$ such that
and constants $\kappa>1$, $a_{0},\eta>0$, $C_{1},C_{2}<\infty$ with the following properties.
\begin{enumerate}
\item $\forall k\in\Z \quad \mu(X\setminus\cup_{Q\in\calD_{k}} Q)=0$,
\item If $l\geq k$, $Q\in\calD_{l}$, $Q'\in\calD_{k}$, then either $Q'\subseteq Q$ or $Q'\cap Q=\emptyset$,
\item For every $l\geq k$ and $Q'\in\calD_{k}$ there exists a unique $Q\in\calD_{l}$ such that $Q\supseteq Q'$,\
\item $\forall k\in\Z, Q\in\calD_{k} \quad \exists c_{Q}\in X : B(c_{Q},a_{0}\kappa^{k}) \subseteq Q \subseteq B(c_{Q}, C_{1} \kappa^{k})$.
\end{enumerate}
Abusing the notation we write $\scale(Q)=k$ if $Q\in\calD_{k}$, so that each cube remembers its scale although the same cube (as a set) may appear in other $\calD_{l}$.
We use $\calD$ to denote the disjoint union of $\calD_{k}$.

If in addition the collection $\calD$ has the small boundary property \eqref{eq:small-boundary}, then we call $\calD$ a \emph{Christ} system of dyadic cubes.
\end{definition}

\begin{theorem}[\cite{MR1096400}]
\label{thm:christ-cubes}
Every space of homogeneous type admits a system of Christ dyadic cubes.
\end{theorem}
\begin{remark}
It is known that for every quasimetric $\rho$ there exists $0<\alpha\leq 1$ and a metric $d$ such that $d(x,y) \sim \rho(x,y)^{\alpha}$, see \cite{MR546295} for the first proof and \cite{MR2538591} for the fact that one can choose $(2\Cref{qm})^{\alpha}=2$ (a much more verbose statement of the latter result also appears in \cite{MR2987059}).
It is easy to see that a system of (Christ) dyadic cubes with respest to $d$ is also a system of (Christ) dyadic cubes with respect to $\rho$ (with different constants).
However, the small boundary property of dyadic cubes seems to be substantially easier to achieve in the metric setting using the Hardy--Littlewood maximal function on $(0,\infty)$ as on \cite[p.~146]{MR1768535}.
\end{remark}

Throughout the article we fix a system of Christ dyadic cubes $\calD$ on $(X,\rho,\mu)$.
We denote by $\EE_{k}$ the conditional expectation operator onto the $\sigma$-algebra generated by $\calD_{k}$.
The martingale difference operator is denoted by $\DD_{k}=\EE_{k}-\EE_{k+1}$.
The function $\DD_{k}f$ is constant on each cube of scale $k$ and has integral $0$ on each cube of scale $k+1$.

\begin{definition}
A finite family $(\calD^{\alpha})_{\alpha}$ of systems of dyadic cubes is called \emph{adjacent} if all constants in their definitions coincide and there exists a constant $C_{3}<\infty$ such that for every ball $B(x,r)\subset X$ there exists $\alpha$ and a cube $Q\in\calD^{\alpha}$ such that $B(x,r)\subseteq Q \subseteq B(x,C_{3}r)$.
\end{definition}
\begin{theorem}[\cite{MR3113086}]
\label{thm:christ-cubes}
Every space of homogeneous type admits a finite collection of adjacent systems of dyadic cubes.
\end{theorem}
\begin{remark}
It is also possible to construct adjacent systems of Christ dyadic cubes.
\end{remark}
\begin{example}
Let $X=\R^{D}$ with the Euclidean distance and the Lebesgue measure.
For each $\alpha\in\Set{0,1,2}^{D}$ the corresponding \emph{shifted system of dyadic cubes} is given by
\[
\cD^{\alpha} = \Set{ 2^{-k}([0,1)^{D} + m + (-1)^{k}\frac13 \alpha), k\in\Z, m\in\Z^{D}}.
\]
Then the systems $\calD^{\alpha}$, $\alpha\in\Set{0,1,2}^{D}$, are adjacent. 
In fact, on $\R^{D}$ one can construct $D+1$ shifted systems of dyadic cubes that are adjacent \cite{MR1993970}.
\end{example}

\subsection{Sparse collections and operators}
\begin{definition}
Let $\cD$ be a system of dyadic cubes.
A collection $\sparse\subset\cD$ is called
\begin{enumerate}
\item \emph{$\eta$-sparse} if there exist pairwise disjoint subsets $E(Q)\subset Q\in\sparse$ with $\meas{E(Q)} \geq \eta \meas{Q}$ and
\item \emph{$\Lambda$-Carleson} if one has $\sum_{Q'\subset Q, Q'\in\sparse} \mu(Q') \leq \Lambda \mu(Q)$ for all $Q\in\cD$.
\end{enumerate}
\end{definition}
It is known that a collection is $\eta$-sparse if and only if it is $1/\eta$-Carleson \cite[\textsection 6.1]{arXiv:1508.05639}.

The corresponding \emph{sparse operator} is given by
\begin{equation}
\label{eq:sparse-operator}
A_{\sparse}f =  \sum_{Q\in\sparse} \one_{Q} \mean[CQ]{\abs{f}}
\end{equation}
and \emph{sparse square function} by
\begin{equation}
\label{eq:sparse-square}
A_{\sparse}^{2}f = \Big(\sum_{Q\in\sparse} \one_{Q} \mean[CQ]{\abs{f}}^{2}\Big)^{1/2}
\end{equation}
The sparse operators \eqref{eq:sparse-operator} and square functions \eqref{eq:sparse-square} can be dominated by finite linear combinations of similar sparse operators/square functions  with respect to adjacent dyadic grids in which the averages of $f$ are taken over $Q$ instead of $CQ$, cf.\ \cite[Remark 4.3]{MR3484688}.
Hence the usual estimates for sparse operators \cite{MR3778152,MR3455749} apply to \eqref{eq:sparse-operator} and \eqref{eq:sparse-square}.

We say that an operator $T$ is pointwise controlled by a sparse operator (or square function) with constant $C<\infty$ if for every function $f$ there exist $1/2$-sparse collections $\calS^{n}\subset\calD$, $n\in\N$, such that
\[
\abs{Tf}
\leq
C \liminf_{n\to\infty} A_{\calS^{n}}f
\quad\text{or}\quad
\abs{Tf}
\leq
C \liminf_{n\to\infty} A_{\calS^{n}}^{2}f,
\]
respectively, holds pointwise almost everywhere.

We use a version of the sparse domination principle in \cite{MR3484688}.
Given a dyadic grid $\calD$ denote by $\calDp$ the set of pairs $(Q',Q)\in\calD\times\calD$ with $Q'\subseteq Q$.
For a function $F : \calDp \to [0,\infty]$ let
\begin{equation}
\label{eq:Nop}
\Nop_{Q}F(x) := \sup_{x\in Q'\subseteq Q} F(Q',Q),
\quad
\Nop F(x) := \sup_{x\in Q'\subseteq Q} F(Q',Q).
\end{equation}
In the first definition above the supremum is taken over all $Q'$ and in the second over all $Q',Q$.
We use the convention that the supremum of an empty subset of $[0,\infty]$ is $0$.

A simple stopping time argument using the doubling condition and starting at a scale $k_{0}$ gives the following result.
\begin{lemma}
\label{lem:sparse-domination}
Let $F:\calDp\to [0,\infty]$ be a function that is monotonic in the sense that
\[
Q'''\subseteq Q''\subseteq Q'\subseteq Q
\implies
F(Q'',Q')\leq F(Q''',Q)
\]
and $\ell^{r}$-subadditive for some $0<r<\infty$ in the sense that
\[
Q''\subseteq Q'\subseteq Q
\implies
F(Q'',Q)^{r}\leq F(Q'',Q')^{r}+F(Q',Q)^{r}.
\]
Suppose that for every dyadic cube $Q \in \calD$ we have
\begin{equation}
\label{eq:local-weak(1,1)-bound}
\norm{\Nop_{Q}F}_{1,\infty}
\leq
C \norm{f}_{L^{1}(CQ)}.
\end{equation}
Then there exist $1/2$-sparse collections $\sparse^{k_{0}} \subset \calD$ of cubes such that
\begin{equation}
\label{eq:sparse-domination}
\Nop F
\lesssim
C \liminf_{k_{0}\to \infty} \Big( \sum_{Q\in\sparse^{k_{0}}} \one_{Q} \mean[CQ]{\abs{f}}^{r} \Big)^{1/r}
\end{equation}
holds pointwise almost everywhere.
\end{lemma}

\subsection{Bounded $r$-variation and jump counting}

\begin{definition}
Let $I$ be an ordered set and $(a_{t})_{t\in I}$ a family of complex numbers.
The \emph{homogeneous $r$-variation seminorm} is denoted by
\[
\hV{a_{t}}{t\in I} := \sup_{t_{0}<t_{1}<\dots<t_{J} \in I} \big( \sum_{j=1}^{J} \abs{a_{t_{j}} - a_{t_{j-1}}}^{r} \big)^{1/r}
\]
and the \emph{inhomogeneous $r$-variation norm} by
\[
\V{a_{t}}{t\in I} := \sup_{t\in I} \abs{a_{t}} + \hV{a_{t}}{t\in I}.
\]
The \emph{$\lambda$-jump counting function} $\jump{a_{t}}{t\in I}$ is the supremum over all $J$ such that there exist $t_{0}<t_{1}<\dots<t_{J}$ with $\abs{a_{t_{j}}-a_{t_{j-1}}}>\lambda$ for all $j=1,\dots,J$.
\end{definition}
We refer to \cite{MR2434308} for the basic properties of the jump counting function and its relation to $r$-variations.

\section{Short variations}
\label{sec:short}
We will consider two types of short variation operators, both involving a non-tangentional supremum component introduced in \cite{2013arXiv1309.2336K}.
Ergodic averages will be compared with a dyadic martingale using the short variations
\begin{equation}
\label{eq:short-var-av}
S_{k}f(x) := \sup_{\rho(x,x') \leq C \kappa^{k}} \V[2]{A_{t}f(x')-\EE_{k}f(x')}{\kappa^{k}/C\leq t\leq C\kappa^{k}}.
\end{equation}
For a CZ kernel $K$ as in Theorem~\ref{thm:sing} (in fact we only need the cancellation condition in $y$ and the size estimate, but not the smoothness condition) we define
\begin{equation}
\label{eq:short-var-sing}
S_{k}f(x) := \sup_{\rho(x,x') \leq C \kappa^{k}} \V[2]{ \int_{\kappa^{k}/C < \rho(x',y) < t} K(x',y) f(y) \dif y}{\kappa^{k}/C\leq t\leq C\kappa^{k}}.
\end{equation}
The non-tangentional short variation square function is defined by
\begin{equation}
\label{eq:short-var-sqf}
S f := (\sum_{k} (S_{k} f)^{2})^{1/2}.
\end{equation}
In the short variation we can in fact replace $2$-variation by $r$-variation for any $r>1$, see \cite{arxiv:1409.7120}.
The following result has been essentially proved in \cite{arxiv:1409.7120} following the arguments in \cite{MR1645330,MR1996394}.

\begin{theorem}
\label{thm:short}
Let $S_{k}$ be given either by \eqref{eq:short-var-av} or by \eqref{eq:short-var-sing}.
Then the operator $S$ given by \eqref{eq:short-var-sqf} has weak type $(1,1)$.
\end{theorem}

As in \cite{arXiv:1605.02936}, Theorem~\ref{thm:short} immediately implies that $S$ is pointwise controlled by a sparse square function.
Using the results from \cite{MR3778152} this implies a sharp weighted refinement of the case $p>1$ of \cite[Theorem 1.4]{arxiv:1409.7120}.
Similarly as in \cite{arXiv:1605.02936} one can also show $\norm{Sf}_{L^{2}(w)} \lesssim \norm{f}_{L^{2}(Mw)}$ and $\norm{Sf}_{L^{1,\infty}(w)} \lesssim \norm{f}_{L^{1}(Mw)}$.

\subsection{A reverse Hölder inequality}
\label{sec:reverse-hoelder}
\begin{lemma}[{cf.\ \cite[Lemma 4.2]{MR1996394}}]
\label{lem:reverse-holder-single-scale}
Let $\calQ\subset\calD$ be a collection of disjoint dyadic cubes of scale $\leq k$.
For each $Q\in\calQ$ let $b^{Q}$ be a scalar-valued function supported on $Q$ with $\int b^{Q} = 0$.
Then for every $\alpha> \frac{D-\eta}{2}$ we have
\begin{equation}
\label{eq:V-bi-biQ-single-scale}
S_{k}(\sum_{Q\in\calQ} b^{Q}(x))^{2}
\lesssim
(\kappa^{k})^{-2 D} \sum_{Q : \dist(x,Q) \lesssim \kappa^{k}} (\kappa^{k}/\diam(Q))^{2 \alpha} \norm{b^{Q}}_{1}^{2}.
\end{equation}
\end{lemma}
\begin{proof}
Clearly only cubes with $\dist(x,Q) \lesssim \kappa^{k}$ contribute to the left-hand side, so we may remove all other cubes from $\calQ$.
Now the right-hand side of the conclusion becomes independent of $x$, so it suffices to estimate the variation at an arbitrary point, which we again call $x$.

We consider only the homogeneous variation, in order to get the inhomogeneous variation it suffices to additionally consider an arbitrary (but fixed) $t$, which is similar but easier.
For a suitable sequence $\kappa^{k}/C\leq t_{1}<\dots<t_{J} \leq C \kappa^{k}$ we have
\[
S_{k}(\sum_{Q} b^{Q}(x))^{2}
\lesssim
\sum_{j} \abs{\sum_{Q} \Bop b^{Q}(x,t_{j},t_{j+1})}^{2},
\]
where $\Bop$ is either a difference between two averages or an integral over an annulus.
We decompose $\calQ = \cup_{i\leq 0} \calQ_{i}$ according to scale: $\calQ_{i}=\calQ\cap\calD_{k+i}$.
Each cube $Q$ can only contribute to the $j$-th summand non-trivially in two cases:
\begin{enumerate}
\item if $Q\cap (\partial B(x,t))\neq\emptyset$ for $t=t_{j}$ or $t=t_{j+1}$
\item or $Q \subset B(x,t_{j+1})\setminus B(x,t_{j})$ in the case \eqref{eq:short-var-sing}.
\end{enumerate}
The second case can only occur for one index $j$, and the contribution of this term can be estimated even without the growing factor in \eqref{eq:V-bi-biQ-single-scale}.
In the first case, for a fixed scale $k+i$, $i\leq 0$, the small boundary property implies that there are at most $O(\kappa^{kD}\kappa^{i\eta}/\kappa^{(k+i)D})=O(\kappa^{i(\eta-D)})$ cubes of this kind in $\calQ_{i}$, call this collection $\calQ_{i,j}$.
Splitting the collection $\calQ$ into scales and applying Hölder's inequality we estimate
\begin{multline}
\label{eq:V-bi-biQ:CS-single-scale}
\sum_{j} \abs{\sum_{i\leq 0} \sum_{Q \in \calQ_{i}} \Bop b^{Q}(x, t_{j}, t_{j+1})}^{2}\\
\lesssim
\sum_{j}
\big( \sum_{i\leq 0} \sum_{Q\in\calQ_{i,j}} \abs{\kappa^{-\alpha i} \Bop b^{Q}(x, t_{j}, t_{j+1})}^{2} \big)
\cdot \big(\sum_{i\leq 0} \sum_{Q\in\calQ_{i,j}} \abs{\kappa^{\alpha i}}^{2} \big)\\
\lesssim
\sum_{j}
\sum_{i\leq 0} \sum_{Q\in\calQ_{i}} \abs{\kappa^{-\alpha i} \Bop b^{Q}(x, t_{j}, t_{j+1})}^{2}
\end{multline}
in view of the hypothesis on $\alpha$ and the fact that $\abs{\calQ_{i,j}} \lesssim \kappa^{i(\eta-D)}$.
Estimating the $2$-variation norm by the $1$-variation norm we obtain
\begin{multline*}
\eqref{eq:V-bi-biQ:CS-single-scale} \lesssim
\sum_{i\leq 0} \kappa^{-2 \alpha i} \sum_{Q\in\calQ_{i}} \big( \sum_{j} \Bop b^{Q}(x, t_{j}, t_{j+1}) \big)^{2}\\
\lesssim
\kappa^{-2kD} \sum_{Q : \dist(Q,x) \lesssim \kappa^{k}} (\kappa^{k}/\diam(Q))^{2 \alpha} \norm{b^{Q}}_{1}^{2}.
\end{multline*}
Taking the supremum over sequences $(t_{j})_{j}$ we obtain \eqref{eq:V-bi-biQ-single-scale}.
\end{proof}

\subsection{Strong type $(2,2)$}
\begin{lemma}
\label{lem:short-variation-single-scale}
There exists an $\epsilon>0$ such that for all $k,j\in\Z$ we have
\begin{equation}
\label{eq:short-variation-single-scale:claim}
\norm{S_{k}(\DD_{k+j} f)}_{2}
\lesssim
2^{-\epsilon \abs{j}} \norm{\DD_{k+j} f}_{2}.
\end{equation}
\end{lemma}
\begin{proof}
Rescaling the metric we may assume $k=0$.

Consider first $j< 0$.
Then $\EE_{0}\DD_{j}=0$, and by Lemma~\ref{lem:reverse-holder-single-scale} with $\calQ=\calD_{j+1}$ we obtain
\begin{align*}
S_{0} \DD_{j}f(x)
&\lesssim
\Big(\sum_{Q \in \calD_{j+1} : \dist(x,Q) \lesssim 1} (\kappa^{-j})^{\alpha 2} \norm{\one_{Q} \DD_{j} f}_{1}^{2} \Big)^{1/2}\\
&\lesssim
(\kappa^{-j})^{\alpha} \Big(\sum_{Q \in \calD_{j+1} : \dist(x,Q) \lesssim 1} \norm{\one_{Q} \DD_{j} f}_{2}^{2} \meas{Q}\Big)^{1/2}\\
&\lesssim
\kappa^{j(D/2-\alpha)} \Big(\int_{\dist(x,y)\lesssim 1} \abs{\DD_{j} f(y)}^{2} \Big)^{1/2},
\end{align*}
and the conclusion follows since the averaging operator of scale $1$ is bounded on $L^{1}$ and we can choose $\frac{D-\eta}{2} < \alpha < \frac{D}{2}$.

Consider now $j\geq 0$.
Then $S_{0}$ does not vanish only in $1$-neighborhoods of boundaries of $Q\in\calD_{j}$, and there we can estimate the $2$-variation by the $1$-variation.
In particular,
\begin{align*}
S_{0} \DD_{j}f(x)
\leq
\sum_{Q\in\calD_{j}} S_{0} (\one_{Q}\DD_{j}f)(x)
\lesssim
\sum_{Q\in\calD_{j}} \norm{\one_{Q}\DD_{j}f}_{1} \one_{\partial_{1} Q}(x).
\end{align*}
The latter characteristic functions have bounded overlap in view of the doubling property of our measure space.
Hence
\[
\norm{S_{0} \DD_{j}f}_{2}^{2}
\lesssim
\sum_{Q\in\calD_{j}} \norm{\one_{Q}\DD_{j}f}_{1}^{2} \meas{\partial_{1} Q}
\lesssim
\sum_{Q\in\calD_{j}} \norm{\one_{Q}\DD_{j}f}_{1}^{2} \meas{Q} \kappa^{-j\eta}
=
\kappa^{-j\eta} \norm{\DD_{j}f}_{2}^{2},
\]
where we have used the small boundary property \eqref{eq:small-boundary} of the dyadic cubes.
\end{proof}

\subsection{Weak type \texorpdfstring{$(1,1)$}{(1,1)}}
\begin{proof}[Proof of Theorem~\ref{thm:short}]
It follows from Lemma~\ref{lem:short-variation-single-scale} that $S$ has strong type $(2,2)$.

By homogeneity it suffices to prove
\[
\meas{\Set{ x : S f(x) > 1 }}
\lesssim
\norm{f}_{1}.
\]
We use the Calderón--Zygmund decomposition.
Let $\calQ\subset\calD$ be disjoint cubes such that such that $\norm{f}_{L^{\infty}(X \setminus \cup \calQ)} \leq 1$, $\sum_{Q\in\calQ} \meas{Q} \lesssim \norm{f}_{1}$, and $\norm{f}_{L^{1}(Q)} \lesssim \meas{Q}$ for all $Q\in\calQ$.
Let
\[
g(x) =
\begin{cases}
\meas{Q}^{-1} \int_{Q} f, & x\in Q\in\calQ,\\
f(x), & x\not\in\cup \calQ
\end{cases}
\]
and
\[
b = \sum_{Q\in\calQ} b^{Q},
\quad
b^{Q}(x) =
\begin{cases}
f(x) - \meas{Q}^{-1} \int_{Q} f, & x\in Q,\\
0, & x\not\in Q.
\end{cases}
\]
Then $\norm{g}_{1} \leq \norm{f}_{1}$ and $\norm{g}_{\infty} \lesssim 1$, so we get the required weak type bound for $g$ from the strong type $(2,2)$ bound.

With $E := \cup_{Q\in\calQ} C Q$, where $C$ is sufficiently large in terms of the constants in the definition of $S_{k}$, it suffices to show
\begin{equation}
\label{eq:outside-exceptional}
\meas{\Set{ S b > 1 } \setminus E}
\lesssim
\sum_{Q} \meas{Q}.
\end{equation}
To this end it suffices to show
\begin{equation}
\label{eq:outside-exceptional-Lr}
\int_{X \setminus E} \sum_{k} (S_{k}b)^{2}
\lesssim
\sum_{Q} \meas{Q}.
\end{equation}

Let $(D-\eta)/2<\alpha<D/2$.
For $x\not\in E$ only cubes of scale $\leq k$ contribute to $S_{k}$.
Thus by Lemma~\ref{lem:reverse-holder-single-scale} we have
\[
(S_{k} b)^{2}(x)
\lesssim
(\kappa^{k})^{-2D} \sum_{Q : \dist(x,Q)\lesssim \kappa^{k}} (\kappa^{k}/\diam(Q))^{2\alpha} \norm{b^{Q}}_{1}^{2}.
\]
Hence the left-hand side of \eqref{eq:outside-exceptional-Lr} can be estimated by
\begin{align*}
\MoveEqLeft
\int_{X \setminus E} \sum_{k} (\kappa^{k})^{-2D} \sum_{Q} \one_{\dist(x,Q)\lesssim \kappa^{k}} (\kappa^{k}/\diam(Q))^{2\alpha} \norm{b^{Q}}_{1}^{2} \dif x
\\
&\lesssim
\sum_{Q} \sum_{k\geq \scale(Q)} (\kappa^{k})^{-2D} (\kappa^{k}/\diam(Q))^{2\alpha} \meas{Q}^{2} \int_{X \setminus E}  \one_{\dist(x,Q)\lesssim \kappa^{k}} \dif x\\
&\lesssim
\sum_{Q} \diam(Q)^{-2\alpha} \meas{Q}^{2} \sum_{k\geq \scale(Q)} \kappa^{k(-D+2\alpha)}\\
&\lesssim
\sum_{Q} \meas{Q}^{2} \diam(Q)^{-D}\\
&\lesssim
\sum_{Q} \meas{Q}.
\qedhere
\end{align*}
\end{proof}

\section{Averages}
The estimate for the jump counting function in Theorem~\ref{thm:av} will follow from Lemma~\ref{lem:sparse-domination} and the next result.
\begin{proposition}
\label{prop:av-jump}
For every $f\in L^{1}(X) \cap L^{\infty}(X)$, every $\lambda>0$ and $k_{\max}<\infty$ there exists a subadditive function $F_{\lambda}$ on $\calDp$ such that
\[
\lambda \sqrt{\jump{A_{t}f(x)}{\kappa^{\scale(Q')}\leq t \leq \kappa^{\scale(Q)}}}
\leq
F_{\lambda}(Q',Q),
\quad
x\in Q'\subseteq Q,
\scale(Q) \leq k_{\max},
\]
and
\[
\mu\Set{\Nop_{Q} F_{\lambda} > \nu}
\leq
C \nu^{-1} \norm{f}_{L^{1}(CQ)}
\quad\text{for all}\quad Q\in\calD, \nu>0.
\]
\end{proposition}
\begin{proof}
Fix $\lambda>0$.
Note first that, as in \cite[Lemma 1.3]{MR2434308}, we have
\begin{align*}
\lambda \sqrt{\jump{A_{t}f(x)}{\kappa^{\scale(Q')}\leq t \leq \kappa^{\scale(Q)}}}
&\lesssim
\Big( \sum_{k=\scale(Q')}^{\scale(Q)-1} \sup_{x'\in Q_{k}(x)} S_{k}f(x')^{2} \Big)^{1/2}\\
&+
\lambda \sqrt{\jump[\lambda/2]{\EE_{k}f(x)}{\scale(Q') \leq k \leq \scale(Q)}}.
\end{align*}
The first term on the right-hand side is a subadditive function $F$ of the order interval $[Q',Q]$, and its contribution is estimated by Theorem~\ref{thm:short}.

In order to estimate the second term we follow the proof of L\'epingle's inequality for martingales.
We construct the greedy stopping times with jump size $\lambda/8$ starting with $l_{0}(x) \equiv k_{\max}$: given $l_{j}(x)$ let $l_{j+1}(x)<l_{j}(x)$ be the largest number with $\abs{(\EE_{l_{j+1}}-\EE_{l_{j}})f(x)} > \lambda/8$ (or $-\infty$ if no such number exists).
Then
\[
\lambda \sqrt{\jump[\lambda/2]{\EE_{k}f(x)}{k_{0} \leq k \leq k_{1}}}
\lesssim
\Big(\sum_{j} \abs{(\EE_{t_{j+1}(x)}-\EE_{t_{j}(x)})(\EE_{k_{0}}-\EE_{k_{1}})f(x)}^{2}\Big)^{1/2}.
\]
Let $F_{\lambda}'(Q',Q)$ denote the right-hand side when $x\in Q'\subset Q$, $\scale(Q')=k_{0}$, and $\scale(Q)=k_{1}$.
Then $F_{\lambda}'(Q',Q)$ is clearly a subadditive function of the order interval $[Q',Q]$.
Moreover,
\[
\Nop F
\lesssim
\Big(\sum_{j} \abs{M_{\calD} (\EE_{t_{j+1}}-\EE_{t_{j}})f}^{2}\Big)^{1/2},
\]
where $M_{\calD}$ is the dyadic maximal operator.
Hence the operator $f\mapsto \Nop F_{\lambda}'$ has strong type $(2,2)$ by orthogonality of the operators $\EE_{t_{j+1}}-\EE_{t_{j}}$ and the martingale maximal inequality.
Standard CZ theory can be applied to show that it also has weak type $(1,1)$ because the filtration $\calD$ is doubling.
We obtain the claim with $F_{\lambda}=F+F_{\lambda}'$.
\end{proof}

In order to pass to $r$-variations we need the fact that the $p$-th power of the $L^{p,\infty}$ quasimetric is equivalent to a metric for $0<p<1$.
\begin{lemma}[{cf.\ \cite[Lemma 2.3]{MR0241685}}]
\label{lem:weakLp<1}
Let $0<p<1$.
Then
\[
\norm{\sum_{j} g_{j}}_{p,\infty}^{p}
\leq 2^{p} \Big(1+\frac{1}{1-p}\Big)
\sum_{j} \norm{g_{j}}_{p,\infty}^{p}.
\]
\end{lemma}
\begin{proof}
It suffices to consider $g_{j}\geq 0$.
Moreover, by homogeneity we may assume
\[
\sum_{j} c_{j} = 1,
\quad
c_{j} = \norm{g_{j}}_{p,\infty}^{p}.
\]
With this normalization we have to show
\[
\meas{\Set{\sum_{j} g_{j} > \lambda}}
\leq 2^{p} \Big(1+\frac{1}{1-p}\Big)
\lambda^{-p}.
\]
To this end decompose $g_{j} = l_{j} + m_{j} + u_{j}$, where $u_{j} = g_{j} \one_{g_{j} > \lambda/2}$ and $l_{j} = \one_{g_{j} \leq \lambda/2} \min(g_{j}, c_{j}\lambda/2)$.
Then
\[
\sum_{j} \meas{\supp u_{j}}
=
\sum_{j} \meas{\Set{g_{j} > \lambda/2}}
\leq
\sum_{j} c_{j} (\lambda/2)^{-p}
=
(\lambda/2)^{-p}
\]
and
\[
\sum_{j} l_{j}
\leq
\sum_{j} c_{j} \lambda/2
\leq
\lambda/2.
\]
Hence it remains to estimate
\begin{multline*}
\meas{\Set{\sum_{j} m_{j} > \lambda/2}}
\leq
(\lambda/2)^{-1} \int \sum_{j} m_{j}
\leq
(\lambda/2)^{-1} \sum_{j} \int (g_{j}-c_{j}\lambda/2) \one_{c_{j}\lambda/2 \leq g_{j} \leq \lambda/2}\\
=
(\lambda/2)^{-1} \sum_{j} \int_{t=c_{j}\lambda/2}^{\lambda/2} \meas{\Set{g_{j}>t}} \dif t
\leq
(\lambda/2)^{-1} \sum_{j} \int_{t=c_{j}\lambda/2}^{\lambda/2} c_{j} t^{-p} \dif t\\
\leq
(1-p)^{-1} (\lambda/2)^{-1} \sum_{j} c_{j} (\lambda/2)^{1-p}
\leq
\frac{2^{p}}{1-p} \lambda^{-p}.
\qedhere
\end{multline*}
\end{proof}

\begin{corollary}
\label{cor:av-Vr}
There exists $C<\infty$ such that for all $r>2$ we have
\[
\mu\Set{\Nop_{Q} F > \nu}
\leq
C \frac{r}{r-2} \nu^{-1} \norm{f}_{L^{1}(CQ)}
\quad\text{for all}\quad Q\in\calD,
\]
where
\[
F(Q',Q) = \sup_{x'\in Q'} \hV{A_{t}f(x')}{\kappa^{\scale(Q')} \leq t \leq \kappa^{\scale(Q)}}.
\]
\end{corollary}
Corollary~\ref{cor:av-Vr} and Lemma~\ref{lem:sparse-domination} imply the estimate for $r$-variations in Theorem~\ref{thm:av}.
\begin{proof}
By homogeneity we may normalize $\nu=1$.
Let $F_{\lambda}$ be the functions from Proposition~\ref{prop:av-jump}.

Suppose $F(Q,Q')>1$.
Then either $F_{1}(Q,Q')>1$ or all jumps involved in the definition of $r$-variation in $F(Q,Q')$ are bounded by $1$.
In the latter case
\[
1
<
F(Q,Q')^{r}
\lesssim
\sum_{l\geq 0} 2^{-lr} \mathcal{J}_{2^{-l}},
\]
so that in either case we obtain
\[
\sum_{l\geq 0} 2^{-l(r-2)} F_{2^{-l}}(Q',Q)^{2} \gtrsim 1.
\]
Hence, using Lemma~\ref{lem:weakLp<1}, we obtain
\begin{align*}
\mu\Set{\Nop_{Q} F > 1}
&\leq
\mu\Set{\Nop_{Q} \big( \sum_{l\geq 0} 2^{-l(r-2)} F_{2^{-l}}^{2} \big) \gtrsim 1}\\
&\lesssim
\norm{ \sum_{l\geq 0} 2^{-l(r-2)} (\Nop_{Q} F_{2^{-l}})^{2} }_{1/2,\infty}^{1/2}\\
&\lesssim
\sum_{l\geq 0} \norm{ 2^{-l(r-2)} (\Nop_{Q} F_{2^{-l}})^{2} }_{1/2,\infty}^{1/2}\\
&=
\sum_{l\geq 0} 2^{-l(r-2)/2} \norm{ \Nop_{Q} F_{2^{-l}} }_{1,\infty}\\
&\lesssim
\norm{f}_{L^{1}(CQ)} \sum_{l\geq 0} 2^{-l(r-2)/2}\\
&\lesssim
\frac{r}{r-2} \norm{f}_{L^{1}(CQ)}.
\qedhere
\end{align*}
\end{proof}

\section{Singular integrals}
\label{sec:singular}
In this section we prove Theorem~\ref{thm:sing}.
Let $\psi$ be a smooth function supported on the interval $[1/\kappa,\kappa]$ with $\sum_{k} \psi(\kappa^{-k}x) = 1$ for all $x\in (0,\infty)$.
Let
\[
K_{k}(x,y) = K(x,y) \psi(\kappa^{-k}\rho(x,y))
\]
be the smoothly truncated dyadic pieces of $K$ of scale $k$ and let $T_{k}$ be the corresponding integral operators
\[
T_{k}f(x) = \int K_{k}(x,y) f(y) \dif y.
\]
This coincides with the notation for truncated singular integrals \eqref{eq:Tt}, but there should be no confusion between integer parameters $k$ and positive real parameters $t$.
The H\"older continuity hypothesis~\eqref{eq:rho-holder} implies that $K_{k}$ satisfies the same smoothness condition \eqref{eq:smoothness} as $K$ (up to a constant factor).
Standard calculations using the cancellation and smoothness conditions show that
\begin{equation}
\label{eq:Tk-alm-orth}
\norm{T_{k'}^{*} T_{k}}_{L^{2}(X)\to L^{2}(X)}+
\norm{T_{k'} T_{k}^{*}}_{L^{2}(X)\to L^{2}(X)}
\lesssim
2^{-\epsilon \abs{k-k'}}.
\end{equation}
Hence by the Cotlar--Stein lemma \cite[p.~280]{MR1232192} the operator $T=\sum_{k\in\Z}T_{k}$ is bounded on $L^{2}(X)$.

\subsection{Short variations}
Fix $r>2$ and define
\[
F(Q',Q)
:=
\sup_{x\in Q'} \hV{\int_{\rho(x,y)>t}K(x,y)f(y) \dif y}{\kappa^{\scale(Q')} \leq t \leq \kappa^{\scale(Q)}}.
\]
The function $F$ is clearly monotonic and subadditive in the order interval $[Q',Q]$ and sublinear in $f$.
Moreover, $F(Q',Q)$ depends only on the values of $f$ on $CQ$.
Hence, in view of Lemma~\ref{lem:sparse-domination}, Theorem~\ref{thm:sing} will follow if we show that the map $f\mapsto\Nop F$ has weak type $(1,1)$ globally.

We split
\begin{align}
\int_{r<\rho(x,y)} K(x,y) f(y) \dif y
&= \label{eq:si-long}
\sum_{k = k_{0}(r)}^{\infty} \int K_{k}(x,y) f(y) \dif y\\
&- \label{eq:si-error}
\int_{\kappa^{k_{0}(r)-1} < \rho(x,y) < \kappa^{k_{0}(r)}} K_{k_{0}(r)}(x,y) f(y) \dif y\\
&+ \label{eq:si-short}
\int_{r < \rho(x,y) < \kappa^{k_{0}(r)}} K(x,y) f(y) \dif y,
\end{align}
where $k_{0}(r) = \ceil{\log_{\kappa} r}$.
The contributions of \eqref{eq:si-error} and \eqref{eq:si-short} are controlled by the square function \eqref{eq:short-var-sing}, so it remains to estimate the contribution of \eqref{eq:si-long}.

\subsection{Strong type $(2,2)$}
Let
\[
F'(Q',Q) := \sup_{x\in Q'} \hV{\sum_{k\geq k_{0}} T_{k}f(x)}{\scale(Q') \leq k_{0} \leq \scale(Q)}.
\]
First we show that the operator $f\mapsto\Nop F'$ has strong type $(2,2)$.
Following \cite[p.~548]{MR837527} we decompose
\[
\sum_{k\geq k_{0}}T_{k}
=
\sum_{l\in\Z, k\geq k_{0}}\DD_{l} T_{k}
=
\sum_{k\in\Z, l\geq k_{0}}\DD_{l} T_{k}
- \sum_{l\geq k_{0}, k<k_{0}}\DD_{l} T_{k}
+ \sum_{l<k_{0}, k\geq k_{0}}\DD_{l} T_{k}
=: I + II + III.
\]
The contribution of $I$ to $\Nop F'(x)$ is bounded by
\[
\hV{\sum_{l\geq k_{0}}\DD_{l}Tf(x)}{k_{0}\in\Z}.
\]
Notice that the supremum over $x\in Q'$ disappeared because functions in the image of $\DD_{l}$ are constant on dyadic cubes of scale $\leq l$.
This operator is bounded on $L^{2}(X)$ with norm $\lesssim r/(r-2)$ since $T$ is bounded on $L^{2}(X)$ and by L\'epingle's inequality for martingales in the form \cite[Lemma 3.3]{MR1019960}.

The remaining two terms will be estimated by square functions.

\subsubsection*{Estimate for $II$}
As in the estimate for $I$ we remove the supremum over $x\in Q'$.
Then we estimate the $r$-variation by the $\ell^{2}$ norm.
Using Minkowski's inequality for the sum over $m$ and Cauchy--Schwarz inequality for the sum over $k$ we estimate
\begin{align*}
\Big( \sum_{k_{0}\in\Z} \abs{\sum_{l\geq k_{0}, k<k_{0}}\DD_{l} T_{k} f}^{2} \Big)^{1/2}
&=
\Big( \sum_{k_{0}\in\Z} \abs{\sum_{m\geq 1} \sum_{k\in\Z} \one_{k<k_{0}\leq k+m} \DD_{k+m} T_{k} f}^{2} \Big)^{1/2}\\
&\leq
\sum_{m\geq 1} \Big( \sum_{k_{0}\in\Z} m \sum_{k\in\Z} \abs{\one_{k<k_{0}\leq k+m} \DD_{k+m} T_{k} f}^{2} \Big)^{1/2}\\
&=
\sum_{m\geq 1} m \Big( \sum_{k\in\Z} \abs{\DD_{k+m} T_{k} f}^{2} \Big)^{1/2}.
\end{align*}
Hence it suffices to find an estimate for the square function on the right-hand side that decays sufficiently quickly with $m$.
By orthogonality
\[
\norm{ \Big( \sum_{k\in\Z} \abs{\DD_{k+m} T_{k} f}^{2} \Big)^{1/2}}_{2}
=
\norm{\sum_{k\in\Z} \DD_{k+m} T_{k} f}_{2}.
\]
For each $k$ we have
\[
\norm{\DD_{k+m} T_{k}}_{L^{2}(X)\to L^{2}(X)}
=
\norm{T_{k}^{*}\DD_{k+m}}_{L^{2}(X)\to L^{2}(X)}
\lesssim
2^{-\epsilon m}
\]
since $T_{k}^{*}\one_{Q}$ is bounded by a universal constant and supported on $\partial_{C\kappa^{k}}Q$ for every $Q\in\calD_{k+m}$ in view of the cancellation condition.

Also, for all $k\neq k'$ we have
\[
(\DD_{k+m} T_{k})^{*} \DD_{k'+m} T_{k'} = 0
\]
and
\[
\norm{\DD_{k'+m} T_{k'} (\DD_{k+m} T_{k})^{*}}_{L^{2}(X)\to L^{2}(X)}
\leq
\norm{T_{k'} T_{k}^{*}}_{L^{2}(X)\to L^{2}(X)}
\lesssim
2^{-\epsilon \abs{k-k'}}
\]
by \eqref{eq:Tk-alm-orth}.
By the Cotlar--Stein lemma \cite[p.~280]{MR1232192} it follows that $\norm{\sum_{k\in\Z} \DD_{k+m} T_{k}}_{L^{2}(X)\to L^{2}(X)} \lesssim 2^{-\epsilon m}$.

\subsubsection*{Estimate for $III$}
This estimate is similar to $II$, but this time we have to keep the supremum over $x'\in Q'$:
\begin{align*}
\MoveEqLeft
\Big( \sum_{k_{0}\in\Z} \sum_{Q\in\calD_{k_{0}}} \one_{Q}(x) \sup_{x'\in Q} \abs{\sum_{l< k_{0} \leq k}\DD_{l} T_{k} f(x')}^{2} \Big)^{1/2}\\
&\leq
\Big( \sum_{k_{0}\in\Z} \sup_{\rho(x,x')\leq C\kappa^{k_{0}}} \abs{\sum_{m\geq 1} \sum_{k\in\Z} \one_{k-m<k_{0}\leq k} \DD_{k-m} T_{k} f(x')}^{2} \Big)^{1/2}\\
&\leq
\Big( \sum_{k_{0}\in\Z} \Big( \sum_{m\geq 1} \sum_{k\in\Z} \one_{k-m<k_{0}\leq k} \sup_{\rho(x,x')\leq C\kappa^{k}} \abs{\DD_{k-m} T_{k} f(x')} \Big)^{2} \Big)^{1/2}\\
&\leq
\sum_{m\geq 1} \Big( \sum_{k_{0}\in\Z} m \sum_{k\in\Z} \Big( \one_{k-m<k_{0}\leq k} \sup_{\rho(x,x')\leq C\kappa^{k}} \abs{\DD_{k-m} T_{k} f(x')} \Big)^{2} \Big)^{1/2}\\
&=
\sum_{m\geq 1} m \Big( \sum_{k\in\Z} \sup_{\rho(x,x')\leq C\kappa^{k}} \abs{\DD_{k-m} T_{k} f(x')}^{2} \Big)^{1/2}
\end{align*}
We view the square sum on the right-hand side as an intrinsic square function.
Indeed, the condition \eqref{eq:rho-holder} implies that $K_{k}$ also has H\"older type regularity.
Notice that $\DD_{k-m}g(x') = \int g h_{k-m,x'}$, where $h_{k-m,x'}$ is a function with mean $0$ and bounded $L^{1}$ norm supported in $B(x,C\kappa^{k-m})$.
It follows that for each $x'$ the function $y\mapsto \DD_{k-m}K_{k}(x',y)$, where $\DD_{k-m}$ acts in the first variable, has absolute value $O(\kappa^{-k D}\kappa^{-\epsilon m})$.
Moreover, it is H\"older continuous since it is an average of H\"older continuous functions.
The required decay in $m$ now comes from the estimate for the intrinsic square function, \cite[Theorem 1.3]{arXiv:1605.02936}, with $\Phi=\kappa^{-\epsilon m}$.

\subsection{Weak type $(1,1)$}
By homogeneity it suffices to show
\[
\meas{\Set{\Nop F' > 1}}
\lesssim
\norm{f}_{1}.
\]
We make a CZ decomposition $f=g+b$.
The good part $g$ is controlled by the $L^{2}$ estimate.
Let $Q\in\calD$ be a bad cube, $b_{Q}$ the corresponding bad function, and $\tilde Q \supset Q$ a ball containing $Q$ with a much larger, but still comparable, radius.
We estimate
\begin{align*}
\int_{X\setminus \tilde Q} \Nop F'(b_{Q})(x) \dif x
&\lesssim
\int_{X\setminus \tilde Q} \sum_{k>\scale(Q)} \sup_{\rho(x,x')\lesssim \kappa^{k}} \abs{T_{k} b_{Q}(x')} \dif x\\&\lesssim
\sum_{k>\scale(Q)} \kappa^{k D} \norm{T_{k} b_{Q}}_{\infty}\\
&\lesssim
\sum_{k>\scale(Q)} \kappa^{-\epsilon(k-\scale(Q))} \norm{b_{Q}}_{1}\\
&\lesssim
\norm{b_{Q}}_{1},
\end{align*}
where we have used that $T_{k}b_{Q}$ is supported in a ball or radius $O(\kappa^{k})$, the mean zero property of $b_{Q}$, and H\"older continuity of $K_{k}$.
Summing over all bad cubes we obtain the claim.

\begin{remark}
The above proof also yields $L^{p}$ and weak type estimates for the jump counting function $\jump{T_{t} f(x)}{t>0}$, and even its localized non-tangentional maximal version.
However, the jump counting function is not subadditive, and unlike in the case of averages we have been unable to construct a subadditive majorant that still satisfies the localized non-tangentional weak type $(1,1)$ estimate.
\end{remark}

\printbibliography
\end{document}